\documentclass[a4paper,10pt]{amsproc}

\usepackage{a4}
\usepackage[utf8]{inputenc}%hiesiger Zeichensatz
\usepackage{fancyhdr}
\usepackage{amsmath}
\usepackage{amssymb}
\usepackage{amsthm}
\usepackage{enumerate}
\usepackage[english]{babel}
\usepackage{lastpage}
\usepackage{multicol}
\usepackage{color}
\usepackage{hyperref}
\usepackage{booktabs}
\usepackage{todonotes}

\title{Cross-Intersecting Erd\H{o}s-Ko-Rado Sets in Finite Classical Polar Spaces}
\author{Ferdinand Ihringer}

\newcommand{\scrR}{\mathcal{R}}
\newcommand{\scrP}{\mathcal{P}}
\newcommand{\bbC}{\mathbb{C}}

\newcommand{\bbR}{\mathbb{R}}

\newcommand{\codim}{\text{codim}}
\newcommand{\PG}{\text{PG}}

\newcommand{\pstype}{e}
\newcommand{\gauss}[2]{{#1 \brack #2}}
\newcommand{\gaussm}[1]{[#1]}

\numberwithin{equation}{section}

\theoremstyle{plain}
\newtheorem{satz}[equation]{Theorem}
\newtheorem{lemma}[equation]{Lemma}
\newtheorem{cor}[equation]{Corollary}
\newtheorem{prop}[equation]{Proposition}
\newtheorem{defi}[equation]{Definition}
\theoremstyle{remark}
\newtheorem{Remark}[equation]{Remark}
\newtheorem{Example}[equation]{Example}



\address{ %
Mathematisches Institut,
Justus Liebig University Giessen,
Arndtstra\ss{}e 2,
35392 Giessen,
Germany.}
\email{Ferdinand.Ihringer@math.uni-giessen.de}

\begin{document}

% \section{Work to do}
% 
% \begin{itemize}
% \item :::::::::::::::::::::::  Important once: :::::::::::::
% \item 
% \item :::::::::::::::::::  now the inline todo's: ::::::::::::::::::
% \input{\jobname.todo_list}
% % \input{\jobnames.todo_list}
% \end{itemize}
% 
% 
% \newoutputstream{todo_list}
% % \openoutputfile{\jobname.todo_list}{todo_list}
% \openoutputfile{\jobname.todo_list}{todo_list}
% \newcommand{\todo}[1]{\addtostream{todo_list}{\noexpand\item #1 at page :\thepage, section:
%  \thesection}
% {\color{red} !!!#1 !!!}
% }

% \listoftodos

\begin{abstract}
  A cross-intersecting Erd\H{o}s-Ko-Rado set of generators of a finite classical polar space is a pair $(Y, Z)$ of sets of generators such that all $y \in Y$ and $z \in Z$ intersect in at least a point. We provide upper bounds on $|Y| \cdot |Z|$ and classify the cross-intersecting Erd\H{o}s-Ko-Rado sets of maximum size with respect to $|Y| \cdot |Z|$ for all polar spaces except Hermitian polar spaces in odd projective dimension.
\end{abstract}

\maketitle

\section{Introduction}

Erd\H{o}s-Ko-Rado sets (\emph{EKR sets}) were introduced by Erd\H{o}s, Ko, and Rado \cite{MR0140419} as a set $Y$ of $k$-element subsets of $\{ 1, \ldots, n\}$ such that the elements of $Y$ pairwise intersect non-trivially. In particular, Erd\H{o}s, Ko, and Rado partially classified all such $Y$ of maximum size.

\begin{satz}[Theorem of Erd\H{o}s, Ko, and Rado]
  Let be $n \geq 2k$. Let $Y$ be an EKR set of $k$-element subsets of $\{ 1, \ldots, n\}$. Then
  \begin{align*}
   |Y| \leq \binom{n-1}{k-1}
  \end{align*}
  with equality for $n>2k$ if and only if $Y$ is set of all $k$-elemental sets containing a fixed element.
\end{satz}

Stronger versions of this theorem were later proven by several authors including the 
famous work by Wilson \cite{wilson_ekr_1984}, Ahlswede and Khachatrian \cite{ahlswede_ekr_1997}.

This theorem for EKR sets was generalized to many structures, including subspaces of projective spaces \cite{MR0382015,MR867648,maarten_ekr_planes} and generators (maximal totally isotropic, respectively, singular subspaces) of polar spaces \cite{MR2755082,maarten_ekr_planes,maarten_ekr_hyperbolic}. In polar spaces the problem is partially open, since the maximum size of EKR sets of generators of $H(2d-1, q^2)$, $d > 3$ odd, is still unknown. To the knowledge of the author the best known upper bound is given in \cite{fi_km_ekr_in_hermitians}.

There exists the following modification of the original problem which generated a lot of interest: a \emph{cross-intersecting EKR set} is a pair $(Y, Z)$ of sets of subsets with $k$ elements of $\{ 1, \ldots, n\}$ such that all $y \in Y$ and $z \in Z$ intersect non-trivially. If one wants to generalize the theorem of Erd\H{o}s, Ko, and Rado to this structure, then the following question arises: how do we measure the size of $(Y, Z)$? There are at least two natural choices. Either one goes for an upper bound for $|Y|+|Z|$ or one tries to find the upper bound for $|Y| \cdot |Z|$. In the set case the first project was pursued in \cite{MR0444483}, while the second one was completed in \cite{MR1008163}. 
Results for vectors spaces are due to Tokushige \cite{tokushige_ev_method}. Again this problem can be generalized to polar spaces, where an \emph{cross-intersecting EKR set of generators} is a pair $(Y, Z)$ of sets of generators such that all $y \in Y$ and $z \in Z$ intersect in at least a point. In this setting this paper is only concerned with an upper bound for $|Y| \cdot |Z|$ and a classification of all cross-intersecting EKR sets reaching this bound.

One additional motivation for this problem is the following: as mentioned before the problem of EKR sets of maximum size in $H(2d-1, q^2)$ is still open for $d > 3$ odd. Let $P$ be a point of $H(2d-1, q^2)$ and let $X$ be an EKR set of $H(2d-1, q^2)$. Furthermore, let $Y$ be the set of generators of $X$ on $P$ and $Z$ the set of generators of $X$ not on $P$. Now in the quotient geometry of $P$ isomorphic to $H(2d-3, q^2)$ the projection of the generators of $Y$ and $Z$ onto the quotient geometry is a cross-intersecting EKR set. So both problems are related.

One last thing to point out is that this work does not provide tight upper bounds for cross-intersecting EKR sets in $H(2d-1, q^2)$ for all $d > 1$. The problem is very similar to the open problem of the maximum size of EKR sets in $H(9, q^2)$. Therefore, it could be reasonable to first solve the problem of the maximum size of cross-intersecting EKR sets in $H(7, q^2)$ and then generalize the technique to EKR sets in $H(9, q^2)$.

\section{Projective Spaces \& Polar Spaces}

We refer to \cite{hirschfeld1998projective} for details on \emph{projective spaces}. A projective space $\PG(n-1, q)$ of \emph{projective dimension} $n-1$ (respectively \emph{vector space dimension} $n$) over the field with $q$ elements has exactly
\begin{align*}
  \gauss{n}{k}_{q} := \prod_{i=1}^{k} \frac{q^{n-i+1}-1}{q^i-1}
\end{align*}
subspaces of (vector space) dimension $k$. We denote the number of points in $\PG(n-1, q)$ by
\begin{align*}
  \gaussm{n}_{q} := \gauss{n}{1}_{q}.
\end{align*}
So we have
\begin{align*}
  \gauss{n}{k}_{q} = \prod_{i=1}^{k} \frac{\gaussm{n-i+1}_{q}}{\gaussm{i}_{q}}.
\end{align*}
We shall write $\gauss{n}{k}$ instead of $\gauss{n}{k}_q$ whenever the choice for $q$ is clear. We will often use
the following analog of the recursive definition of binomial coefficients.
\begin{align}
  \gauss{n+1}{k+1} = \gauss{n}{k+1} + q^{(n-k)} \gauss{n}{k} \label{eq_gaussian_rec}
\end{align}

\begin{Remark}
 All the used eigenvalue formulas are more convenient if we use vector space dimensions
 and not projective dimensions. Consequently, the word \emph{dimension} will always refer to
 the vector space dimension of a subspace.
\end{Remark}

A \emph{polar space} is a incidence geometry with subspaces of dimension from $0$ to $d$ defined by a non-degenerate sesquilinear form or a non-degenerate quadratic form. 
The \emph{finite classical polar spaces} are $Q^+(2d-1, q)$, $Q(2d, q)$, $Q^-(2d+1, q)$, $W(2d-1, q)$, where $q$ is a prime power, $H(2d-1, q)$, and $H(2d, q)$, where $q$ is the square of a prime power.
We refer to \cite{hirschfeld1991general} for details. 
Denote totally isotropic, respectively, singular subspaces of (vector space) dimension $d$ as \emph{generators}. 
Each subspace of (vector space) dimension $d-1$ of a polar space is incident with exactly $q^\pstype+1$ generators, where $\pstype = 0$ for $Q^+(2d-1, q)$, $\pstype = 1/2$ for $H(2d-1, q)$, $\pstype = 1$ for $Q(2d, q)$ and $W(2d-1, q)$, $\pstype = 3/2$ for $H(2d, q)$, and $\pstype = 2$ for $Q^-(2d+1, q)$.
It is well known that a polar space possesses exactly 
\begin{align*}
 \prod_{i=0}^{d-1} (q^{i+\pstype}+1)
\end{align*}
generators and $(q^{d+e-1} + 1) \gaussm{d}$ points (i.e. $1$-dimensional totally isotropic, respectively, singular subspaces).

\section{The Association Scheme of a Polar Space}\label{sec_assoc}

We need some basic properties of an association scheme of generators on a dual polar space of rank $d$ and type $\pstype$. 
A complete introduction to association schemes can be found in \cite[Ch. 2]{brouwer1989distance}.

\begin{defi}
  Let $X$ be a finite set. A $d$-class association scheme is a pair $(X, \scrR)$, where $\scrR = \{R_0, \ldots, R_d\}$ is a set of symmetric binary relations on $X$ with the following properties:
  \begin{enumerate}
    \item $\{ R_0, \ldots, R_d\}$ is a partition of $X \times X$.
    \item $R_{0}$ is the identity relation.
    \item There are numbers $p_{ij}^k$ such that for $x, y \in X$ with $x R_k y$ there are exactly $p_{ij}^k$ elements $z$ with $x R_i z$ and $z R_j y$.
  \end{enumerate}
\end{defi}

The number $n_i := p_{ii}^{0}$ is called the $i$-valency of $R_i$. The total number of elements of $X$ is
\begin{align*}
  n := |X| = \sum_{i=0}^d n_i.
\end{align*}

The relations $R_i$ are described by their adjacency matrices $A_i \in \bbC^{n,n}$ defined by
\begin{align*}
  (A_i)_{xy} = \begin{cases}
                 1 & \text{ if } x R_i y\\
                 0 & \text{ otherwise.}
               \end{cases}
\end{align*}
There exist (e.g. in \cite[p. 45]{brouwer1989distance}) idempotent Hermitian matrices $E_j \in \bbC^{n,n}$ (hence they are positive semidefinite) with the properties
\begin{align*}
\begin{array}{lll}\displaystyle
\sum_{j=0}^d E_j = I, & \hspace*{2cm} &\displaystyle E_{0} = n^{-1} J,\\
\displaystyle A_j = \sum_{i=0}^d P_{ij} E_i, &  &\displaystyle E_j = \frac{1}{n} \sum_{i=0}^d Q_{ij} A_i,
  \end{array}
\end{align*}
where $P = (P_{ij}) \in \bbC^{d+1,d+1}$ and $Q = (Q_{ij}) \in \bbC^{d+1,d+1}$ are the so-called eigenmatrices of the association scheme.

The generators of a polar space define an association scheme if we say that two generators $a$ and $b$ are in relation $R_i$ if and only if $\codim(a \cap b) = i$.
Hence a cross-intersecting EKR set $(Y, Z)$ is a set of vertices such that there are no edges between $Y$ and $Z$ in the (distance-regular) graph associated with $A_d$.
This scheme is \emph{cometric}, so there exists a natural ordering of its $E_j$'s and its eigenspaces $W_j$ \cite[Sec. 2.7, Sec. 9.4]{brouwer1989distance}.
The matrix $P$ can be found in the literature (for example in \cite[Theorem 4.3.6]{vanhove_phd}).
In particular, the eigenvalues of $A_d$ are
\begin{align}
  (-1)^{r} q^{\binom{d-r}{2} + \binom{r}{2} + \pstype(d-r)}\label{eq_ev_disj}
\end{align}
for $r \in \{ 0, 1, \ldots, d \}$.
Here $W_0 = \langle j \rangle$, where $j$ is the all-one vector. Furthermore, notice that
all eigenspaces $W_i$ of an association scheme are pairwise orthogonal (see \cite[Ch. 2]{brouwer1989distance}).

\section{An Algebraic Bound}

We shall apply a technique that was, to the knowledge of the author, first used by Willem H. Haemers in \cite{haermers_ci_ekr_sets_interlacing_bound}.
The author learned about this technique from a paper by Tokushige \cite{tokushige_ev_method}, where he uses a variant of the result based on the work of
Ellis, Friedgut, and Pipel \cite{MR2784326} to prove a result on EKR sets of permutations. 
Let $G$ be a graph with $n$ vertices $\{ 1, \ldots, n\}$. A matrix $A = (a_{xy}) \in \bbC^{n \times n}$ is called \emph{extended weight adjacency matrix} of $G$ if $A$ is symmetric, and
\begin{enumerate}
 \item $a_{xy} \leq 0$ if $x$ and $y$ are non-adjacent,
 \item $a_{xx} = 0$,
 \item the all-ones vector $j$ is an eigenvector of $A$.
 \item $A$ is not the all-zero matrix.
\end{enumerate}

Let $\lambda_{1}, \ldots, \lambda_{n}$ be the (possibly pairwise equal) eigenvalues of $A$. 
Denote the eigenvalue of $j$ by $k$. 
Denote the smallest eigenvalue by $\lambda_{-}$, and the corresponding eigenspace by $V_{-}$.
Denote the largest eigenvalue with eigenvectors not in $\langle j \rangle$ by $\lambda_{+}$. 
Denote the corresponding eigenspace by $V_{+}$ if $k \neq \lambda_{\max}$, and by $\langle j \rangle \perp V_{+}$ if $k = \lambda_{\max}$.
Denote $\max\{ -\lambda_{-}, \lambda_{+}\}$ by $\lambda_{b}$.
We say that $\lambda_b$ is the \emph{second largest absolute eigenvalue} of $A$.

A \emph{characteristic vector} $\chi_Y$ of a subset $Y \subseteq \{ 1, \ldots, n\}$ is defined by
\begin{align*}
 \chi_{i} = \begin{cases}
    1 & \text{ if } i \in Y\\
    0 & \text{ if } i \notin Y
 \end{cases}.
\end{align*}

Ellis, Friedgut, and Pipel used the following result, a generalization of the Hoffman bound for cocliques in graphs:
\begin{lemma}\label{lem_efp}
  Let $A$ be an extended weight adjacency matrix of a regular graph $G$ with $n$ vertices.
  Let $k$ be the eigenvalue of the all-one vector $j$, and let $\lambda_b$ the second
  largest absolute eigenvalue of $A$.
  Let $Y$ and $Z$ be sets of vertices such that there are no edges between $Y$ and $Z$. Then
  \begin{align*}
    \sqrt{|Y| \cdot |Z|} \leq \frac{\lambda_b}{k + \lambda_b} n.
  \end{align*}
\end{lemma}

\begin{Remark}
  Very often the Hoffman bound is only formulated for so-called \emph{weight adjacency matrices} or
  \emph{pseudo adjacency matrices} where $a_{ij}$ is zero if $i$ and $j$ are non-adjacent
  and $a_{ij} > 0$ if $i$ and $j$ are adjacent. It is regularly mentioned in the literature and easy to see
  that all the proofs for variants of the Hoffman bound (at least the ones used in this paper)
  also work for extended adjacency matrices without changing much of the proof.
  The Hoffman bound for an extended weight matrix $A$ is optimal only if $A$ is also
  a weight matrix. Our more general definition will turn out to be more convenient
  in Section \ref{sec_hermitian} where we shall not bother to calculate the exact 
  minimum of the Hoffman bound.
\end{Remark}

Tokushige reformulates Lemma \ref{lem_efp} in a more detailed way in \cite{tokushige_ev_method}, Lemma 2. 
Unfortunately, his reformulation misses to point out some details necessary for the special case handled in this paper. 
Hence, we have to restate his wording of the following lemma.

\begin{lemma}\label{lem_tok}
  Suppose that equality holds in Lemma \ref{lem_efp}.
  Then one of the following cases occurs:
  \begin{enumerate}[(a)]
   \item We have $\lambda_{+} = \lambda_b > -\lambda_{-}$, $\chi_Y = \alpha j + v_+$, and $\chi_Z = \alpha j - v_+$ for some vector $v_+ \in V_+$.
   \item We have $\lambda_{+} < \lambda_b = -\lambda_{-}$, $\chi_Y = \alpha j + v_-$, and $\chi_Z = \alpha j + v_-$ for some vector $v_- \in V_-$. In this case $Y=Z$, and $Y$ is an EKR set.
   \item We have $\lambda_{+} = \lambda_b = -\lambda_{-}$, $\chi_Y = \alpha j + v_- + v_+$, and $\chi_Z = \alpha j + v_- - v_+$ for some vectors $v_- \in V_-$ and $v_+ \in V_+$.
  \end{enumerate}
  Furthermore, $|Y|=|Z| = \alpha n$.
\end{lemma}
\begin{proof}
  The proof of Lemma 2 in \cite{tokushige_ev_method} still works for the first three claims if one reads it carefully.
  For the claim $|Y| = |Z|$ consider the following: $j$ is orthogonal to to $v_-$ and $v_+$. 
  Hence,
  \begin{align*}
    |Y| = \chi_Y^T j = \alpha n = \chi_Z^T j = |Z|.
  \end{align*}
\end{proof}
 
\section{Cross-intersecting EKR Sets of Maximum Size}

In this section we shall calculate tight upper bounds for all polar spaces except
$H(2d-1, q^2)$, and classify all examples in case of equality.
For all polar spaces except $H(2d-1, q^2)$ we can imitate the approach of Pepe, Storme, and Vanhove \cite{MR2755082}.
Recall from Section \ref{sec_assoc} that we have a natural ordering of the 
eigenspaces $W_0 (= \langle j \rangle), W_1, \ldots, W_d$ of the association scheme which
we defined on generators of a polar space of rank $d$.

\begin{satz}\label{thm_es}
 Let $(Y, Z)$ be a cross-intersecting EKR set of generators of a polar space $\scrP$. 
 Let $n$ be the number of generators of $\scrP$.
 Then we have the following:
 \begin{itemize}
  \item If $\scrP = Q^+(2d-1, q)$, then $\sqrt{|Y| \cdot |Z|}$ is at most $n/2$, and if this bound is reached, then $\chi_Y, \chi_Z \in W_{0} \perp W_d$.
  \item If $\scrP \in \{ Q(2d, q), W(2d-1, q)\}$, then $\sqrt{|Y| \cdot |Z|}$ is at most the number of generators on a fixed point, and if this bound is reached, then $\chi_Y, \chi_Z \in W_{0} \perp W_1 \perp W_d$.
  \item If $\scrP \in \{ H(2d, q), Q^-(2d+1, q) \}$, then $\sqrt{|Y| \cdot |Z|}$ is at most the number of generators on a fixed point, and if this bound is reached, then $\chi_Y, \chi_Z \in W_{0} \perp W_1$.
 \end{itemize}
\end{satz}
\begin{proof}
  To apply Lemma \ref{lem_efp}, we have to calculate the second largest absolute
  eigenvalue of the disjointness graph (with associated adjacency matrix $A_d$).
  The eigenvalues of $A_d$ were given in \eqref{eq_ev_disj} as
  \begin{align*}
   (-1)^{r} q^{\binom{d-r}{2} + \binom{r}{2} + \pstype(d-r)}.
  \end{align*}
  For $r=0$ this is the eigenvalue which belongs to the all-one vector $j$, so
  with $k$ defined as in Lemma \ref{lem_efp} we have
  \begin{align*}
    k = q^{\binom{d}{2} + d\pstype}.
  \end{align*}
  For $\pstype = 0$ note that the absolute eigenvalues for $r=0$ and $r=d$ are equal. Therefore,
  the eigenspace belonging to $k$ has dimension at least $2$ which make $k$ also the second
  largest absolute eigenvalue. Hence, we have the following for the different polar spaces.
  For $\pstype = 0$ (i.e. $\scrP = Q^+(2d-1, q)$) the second largest absolute 
  eigenvalue occurs if and only if $r = d$, for $\pstype = 1$ (i.e.  $\scrP \in \{ Q(2d, q), W(2d-1, q)\}$)
  the second largest absolute eigenvalue occurs if and only if $r \in \{ 1, d\}$, for
  $\pstype \in \{ 3/2, 2\}$ (i.e.  $\scrP \in \{ H(2d, q), Q^-(2d+1, q) \}$) the
  second largest absolute eigenvalue occurs if and only if $r = 1$. Applying Lemma \ref{lem_efp}
  and Lemma \ref{lem_tok} yields the assertion.
\end{proof}

Using Lemma \ref{lem_tok} and the classification of EKR sets of generators given in \cite{MR2755082} we get the following result.

\begin{cor}
  Let $(Y, Z)$ be an cross-intersecting EKR set of a finite classical polar space $\scrP$ not isomorphic to $Q(2d, q)$ with $d$ even, $W(2d-1, q)$ with $d$ even, $Q^+(2d-1, q)$ with $d$ even, or $H(2d, q^2)$, where $|Y| \cdot |Z|$ reaches the bound in Theorem \ref{thm_es}. Then $Y = Z$, and $Y$ is an EKR set.
\end{cor}
\begin{proof}
  For the stated cases, the eigenspaces given in Theorem \ref{thm_es} 
  (not equal to $\langle j \rangle$) belong to
  negative eigenvalues. By Lemma \ref{lem_tok}, then all cross-intersecting EKR which reach the bound
  given in Theorem \ref{thm_es} are EKR sets.
\end{proof}

Similar to \cite{MR2755082} we shall continue to classify the more complicated cases.

\subsection{The Hyperbolic Quadric, $d$ even}

The generators of $Q^+(2d-1, q)$ can be partitioned into two sets $X_1$ and $X_2$ of generators (commonly known as \emph{latins} and \emph{greeks}) with $|X_1| = |X_2| = n/2$. 
For $x_1 \in X_1$ and $x_2 \in X_2$ the codimension of the intersection of $x \cap y$ is odd. 
For $x_1, x_2 \in X_1$ the codimension of the intersection of $x \cap y$ is even. 
This implies for $d$ even that $(X_1, X_2)$ is a cross-intersecting EKR set of maximum size according to Theorem \ref{thm_es}. 
There exist $x_1, x_2 \in X_1$ with $\dim(x_1 \cap x_2) = 0$ if $d$ is even, so $(X_1, X_1)$ is not a cross-intersecting EKR set.

\begin{satz}\label{thm_ci_ekr_qp}
  Let $(Y, Z)$ be an cross-intersecting EKR set of maximum size of $Q^+(2d-1, q)$ with $d$ even. Then $Y = X_i$ and $Z=X_j$ for $\{ i, j\} = \{1, 2\}$.
\end{satz}
\begin{proof}
  By Theorem \ref{thm_es}, we have $\chi_Y, \chi_Z \in W_{0} \perp W_d$. 
  As in Theorem 16 of \cite{MR2755082} $W_{0}$ is spanned by $\chi_{X_1}+\chi_{X_2}$, and $W_d$ is spanned by $\chi_{X_1}-\chi_{X_2}$. 
  Hence $\chi_Y, \chi_Z \in \{ \chi_{X_1}, \chi_{X_2}\}$ as $\chi_Y, \chi_Z, \chi_{X_1}, \chi_{X_2}$ are $0$-$1$-vectors with $\chi_{X_1} + \chi_{X_2} = j$. 
  Hence without loss of generality $Y = X_1$. Since $(X_1, X_1)$ is not a cross-intersecting EKR set, we have $Z = X_2$.
\end{proof}

\subsection{The Parabolic Quadric and the Symplectic Polar Space, $d$ even}

If a cross-intersecting EKR set $(Y, Z)$ of $Q(2d, q)$ satisfies $\chi_Y, \chi_Z \in W_{0} \perp W_1$, then $Y = Z$, and $Y$ is an EKR set as before.
So only the case $\chi_Y, \chi_Z \in W_{0} \perp W_1 \perp W_d$ remains. In the following denote $W_1$ by $V_-$ and $W_d$ by $V_+$.
Furthermore, as in Lemma \ref{lem_tok} we write 
\begin{align*}
 \chi_Y &= \alpha j + v_- + v_+ \\
 &= \frac{|Y|}{n} j + v_- + v_+ \\
 &= \frac{\lambda_b}{k+\lambda_b} j + v_- + v_+
 \intertext{ and }
 \chi_Z &= \alpha j + v_- + v_+ \\
 &= \frac{|Z|}{n} j + v_- + v_+ \\
 &= \frac{\lambda_b}{k+\lambda_b} j + v_- - v_+
\end{align*}
with $v_- \in V_-$ and $v_+ \in V_+$.
We need the following well-known lemma.

\begin{lemma}\label{lem_nbr}
  Let $\chi \in \langle j \rangle \perp V$ for some eigenspace $V$ of an (extended weight) adjacency matrix of a $k$-regular graph with $n$ vertices associated with eigenvalue $\lambda$. 
  Then the characteristic vector $e_i$ of the $i$-th vertex satisfies
  \begin{align*}
    e_i^T A \chi = \frac{\chi^T j}{n} (k - \lambda) + \lambda e_i^T \chi.
  \end{align*}
\end{lemma}
\begin{proof}~
  As $\chi \in \langle j \rangle \perp V$, we can write $\chi = \alpha j + v$ for some $v \in V$ and $\alpha = \frac{\chi^T j}{n}$. Then
  \begin{align*}
   e_i^T A \chi &= e_i^T A (\alpha j + v)\\
   &= e_i^T (\alpha k j + \lambda v)\\
   &= e_i^T (\alpha (k-\lambda) j + \lambda \chi)\\
   &= \frac{\chi^T j}{n} (k - \lambda) + \lambda e_i^T \chi.
  \end{align*}
\end{proof}

\begin{cor}\label{lem_inmbs}
  Let $\chi, \psi \in \langle j \rangle \perp V_- \perp V_+$ for some eigenspaces $V_-$, respectively, $V_+$ of a (extended weight) adjacency matrix of the graph with eigenvalue $\lambda_-$, respectively, $\lambda_+$.
  If $\chi = \alpha j + v_- + v_+$ and $\psi = \alpha j + v_- - v_+$ for some $\alpha \in \bbR$, $v^- \in V_-$, and $v^+ \in V^+$, then
  \begin{align*}
    &e_i^T A \chi = \frac{(\chi+\psi)^T j}{2n} (k - \lambda_-) + \frac{(\lambda_-+\lambda_+)}{2} e_i^T \chi\\
    &e_i^T A \psi = \frac{(\chi+\psi)^T j}{2n} (k - \lambda_-) + \frac{(\lambda_- - \lambda_+)}{2} e_i^T \psi
  \end{align*}
\end{cor}
\begin{proof}
  We have $\chi + \psi \in \langle j \rangle \perp V_-$ and $\chi - \psi \in V_+$. By Lemma \ref{lem_nbr} and $j^T v^- = 0 = j^T v^+$,
  \begin{align*}
    &e_i^T A (\chi + \psi) = \frac{(\chi+\psi)^T j}{n} (k - \lambda_-) + \lambda_- e_i^T (\chi+\psi)\\
    &e_i^T A (\chi - \psi) = \lambda_+ e_i^T (\chi-\psi).
  \end{align*}
  Now the equations $2 e_i^T A \chi = e_i^T A (\chi + \psi) + e_i^T A (\chi - \psi)$ and $2 e_i^T A \psi = e_i^T A (\chi + \psi) - e_i^T A (\chi - \psi)$ yield the assertion.
\end{proof}

\begin{lemma}\label{lem_ev_q0}
  For the adjacency matrix $A_{d-s}$, $0 < s < d$, the eigenspace $W_1$ is associated with eigenvalue
  \begin{align*}
    \lambda_{-,s} := -\gauss{d-1}{s} q^{\binom{d-s}{2}} + \gauss{d-1}{s-1} q^{\binom{d-s+1}{2}},
  \end{align*}
  the eigenspace $W_d$ is associated with eigenvalue
  \begin{align*}
    \lambda_{+,s} := (-1)^{d-s} \gauss{d}{s} q^{\binom{d-s}{2}},
  \end{align*}
  and 
  \begin{align*}
   k_s := \gauss{d}{s} q^{\binom{d-s+1}{2}} = \left( \gauss{d-1}{s} + \gauss{d-1}{s-1} q^{d-s} \right) q^{\binom{d-s+1}{2}}.
  \end{align*}
\end{lemma}
\begin{proof}
  See \cite[Theorem 4.3.6]{vanhove_phd}, where
  the eigenvalue of $W_j$ for $A_{i}$ is given by
  \begin{align*}
    \sum_{0, j-i \leq u \leq  d-i,j} (-1)^{j+u} \gauss{d-j}{d-i-u} \gauss{j}{u} q^{(u+i-j)(u+i-j+2\pstype -1)/2 + \binom{j-u}{2}}.
  \end{align*}
  For $j=1$, respectively, $j=d$, and $i=d-s$, this formula yields the assertion.
  The last equality is an application of \eqref{eq_gaussian_rec}.
\end{proof}

\begin{prop}\label{prop_q0_inmbs}
  Let $(Y, Z)$ be a cross-intersecting EKR set of $Q(2d, q)$ or $W(2d-1, q)$, $d$ even, of maximum size such that $Y \cap Z \neq Y$. Let $G \in Y$.
  \begin{enumerate}[(a)]
   \item If $d-s$ is even, then $G$ meets $0$ elements of $Z$ in dimension $s$.
   \item If $d-s$ is odd, then $G$ meets $0$ elements of $Y$ in dimension $s$.
   \item If $d-s$ is even, then $G$ meets 
   \begin{align*}
      \gauss{d}{s} q^{\binom{d-s}{2}}
   \end{align*}
   elements of $Y$ in dimension $s$.
   \item If $d-s$ is odd, then $G$ meets
   \begin{align*}
      \gauss{d}{s} q^{\binom{d-s}{2}}
   \end{align*}
   elements of $Z$ in dimension $s$.
   \end{enumerate}
   In particular, $Y \cap Z = \emptyset$.
\end{prop}
\begin{proof}
  We can calculate these numbers with Lemma \ref{lem_ev_q0} and Corollary \ref{lem_inmbs} by choosing $\chi_{\{G\}}$ as $e_i$. 
  For $A_{d-s}$ the parameters are given by
  \begin{align*}
   k_s - \lambda_{-,s} &= \left( \gauss{d-1}{s} + \gauss{d-1}{s-1} q^{d-s} \right) q^{\binom{d-s+1}{2}} \\
   &+ \gauss{d-1}{s} q^{\binom{d-s}{2}} - \gauss{d-1}{s-1} q^{\binom{d-s+1}{2}}\\
   &= q^{\binom{d-s}{2}} \gauss{d-1}{s} \left( q^{d-s} +1 \right) + q^{\binom{d-s+1}{2}} \gauss{d-1}{s-1} \left( q^{d-s} -1 \right) \\
   &\stackrel{\text{Def.}}{=} q^{\binom{d-s}{2}} \gauss{d-1}{s} \left( q^{d-s} +1 \right) + q^{d-s} \cdot q^{\binom{d-s}{2}} \gauss{d-1}{s} \left( q^s -1 \right) \\
   &=q^{\binom{d-s}{2}} \gauss{d-1}{s} \left( q^{d}+1 \right),
   \intertext{for $d-s$ even}
   & \lambda_{-,s} + \lambda_{+,s} \stackrel{\eqref{eq_gaussian_rec}}{=} 2 \gauss{d-1}{s-1} q^{\binom{d-s+1}{2}},\\
   & \lambda_{-,s} - \lambda_{+,s} \stackrel{\eqref{eq_gaussian_rec}}{=} -2 \gauss{d-1}{s} q^{\binom{d-s}{2}},
   \intertext{for $d-s$ odd}
   & \lambda_{-,s} + \lambda_{+,s} \stackrel{\eqref{eq_gaussian_rec}}{=} -2 \gauss{d-1}{s} q^{\binom{d-s}{2}},\\
   & \lambda_{-,s} - \lambda_{+,s} \stackrel{\eqref{eq_gaussian_rec}}{=} 2 \gauss{d-1}{s-1} q^{\binom{d-s+1}{2}}.
   \intertext{By assumption $(Y, Z)$ is of maximum size, so by Lemma \ref{lem_tok} (recall $\lambda_+ = \lambda_{+,0}$, and $\lambda_- = \lambda_{-,0}$)}
   &\chi_Z^T j = |Z| = |Y| = \chi_Y^T j = \frac{n\lambda_+}{k+\lambda_+}.
   \intertext{Hence,}
   \frac{(\chi_Y+\chi_Z)^T j}{2n} &= \frac{\lambda_+}{k + \lambda_+}
   = \frac{q^{\binom{d}{2}}}{q^{\binom{d+1}{2}} + q^{\binom{d}{2}}}
   = \frac{1}{q^{d}+1}.
   \intertext{Hence by Corollary \ref{lem_inmbs},}
   e_i^T A \chi &= \frac{k_s - \lambda_{-,s}}{q^d+1} + \frac{(\lambda_-+\lambda_+)}{2} e_i^T \chi\\
   &= q^{\binom{d-s}{2}} \gauss{d-1}{s} + \frac{(\lambda_-+\lambda_+)}{2} e_i^T \chi.
   \intertext{If $d-s$ is even and $e_i^T \chi = 1$, then by \eqref{eq_gaussian_rec}}
   e_i^T A \chi &= q^{\binom{d-s}{2}} \gauss{d-1}{s} + q^{\binom{d-s+1}{2}} \gauss{d-1}{s-1} = q^{\binom{d-s}{2}} \gauss{d}{s}.
   \intertext{If $d-s$ is odd and $e_i^T \chi = 1$, then }
   e_i^T A \chi &= q^{\binom{d-s}{2}} \gauss{d-1}{s} - q^{\binom{d-s}{2}} \gauss{d-1}{s} = 0.
  \end{align*}
  All the remaining cases are either similar or trivial.
\end{proof}

Now we have a strong combinatorial information about cross-intersecting EKR sets $(Y, Z)$ of maximum size which are not EKR sets.
By adding some geometrical arguments this leads to a complete classification of cross-intersecting EKR sets in these parabolic and symplectic polar spaces
as we shall see in the following.

\begin{lemma}\label{lem_nmb_dm1z}
  Let $(Y, Z)$ be a cross-intersecting EKR set of $Q(2d, q)$ or $W(2d-1, q)$, $d$ even, of maximum size. 
  Let $G, H \in Y$ disjoint (see Proposition \ref{prop_q0_inmbs} (c)). Let $\pi_1, \ldots, \pi_{\gaussm{d}} \subseteq G$ be the $\gaussm{d}$ subspaces of dimension $d-1$ of $G$.
  Then the following holds:
  \begin{enumerate}[(a)]
   \item Exactly $\gaussm{d}$ elements $z_1, \ldots, z_{\gaussm{d}}$ of $Z$ meet $G$ in dimension $d-1$.\label{lem_nmb_dm1z_a}
   \item We have $\{ z_i ~|~ i \in \{ 1, \ldots, \gaussm{d}\} \} = \{ \langle \pi_i, \pi_i^\perp \cap H \rangle ~|~ i \in \{ 1, \ldots, \gaussm{d}\} \}$.
  \end{enumerate}
\end{lemma}
\begin{proof}
  By Proposition \ref{prop_q0_inmbs} (d) and $Y \neq Z$, exactly $\gaussm{d}$ elements of $Z$ meet $Y$ in dimension $d-1$. This shows \eqref{lem_nmb_dm1z_a}.
  By Proposition \ref{prop_q0_inmbs} (b), $\dim(z_i \cap z_j) < d-1$ for $i \neq j$. 
  Hence, each hyperplane $\pi_i$ lies in exactly one element $z_j$, and $z_j$ satisfies $z_j \subseteq \pi_i^\perp$.
  Since $(Y, Z)$ is a cross-intersecting EKR set, all $z_j$ meet $H$ in at least a point.
  Since $\pi_i \subseteq G$ and $G \cap H = \emptyset$, we see that $\pi_i^\perp \cap H$ is a point.
  Hence,
  \begin{align*}
   \{ z_i ~|~ i \in \{ 1, \ldots, \gaussm{d}\} \} = \{ \langle \pi_i, \pi_i^\perp \cap H \rangle ~|~ i \in \{ 1, \ldots, \gaussm{d}\} \}.
  \end{align*}
\end{proof}
  
  We write
  \begin{align}
    Z_{G, H} &= \{ \{ z_i ~|~ i \in \{1, \ldots, \gaussm{d}\}\} \label{eq_lem_nmb_dm1z}\\
    &= \{ \langle \pi_i, \pi_i^\perp \cap H \rangle ~|~ i \in \{ 1, \ldots, \gaussm{d}\} \} \subseteq Z\notag
  \end{align}
  for $G, H \in Y$ whenever Lemma \eqref{lem_nmb_dm1z} in applicable.

\subsubsection{The Parabolic Quadric $Q(2d, q)$}

Let $h$ be a subspace of $\PG(2d, q)$. We write $Y \subseteq h$ if all elements of $Y$ are subspaces of $h$, $Y \cap h$ for all elements of $Y$ in $h$, and $Y \setminus h$ for all elements of $Y$ not in $h$.
\begin{lemma}\label{lem_q0inhyp_pre2}
  Let $G$ and $H$ be disjoint generators of $Q(2d, q)$.
  Then $h := \langle G, H \rangle \cap Q(2d, q)$ is isomorphic to $Q^+(2d-1, q)$.
\end{lemma}
\begin{proof}
  The generators $G$ and $H$ are disjoint, hence $h \cap Q(2d, q)$ is not degenerate.
  The hyperplane $h$ obviously contains generators, hence $h \cap Q(2d, q)$ does not have type $Q^-(2(d-1)+1, q)$.
  Therefore, the intersection $h \cap Q(2d, q)$ is isomorphic to $Q^+(2d-1, q)$.
\end{proof}

\begin{lemma}\label{lem_q0inhyp_pre1}
  Let $(Y, Z)$ be a cross-intersecting EKR set of $Q(2d, q)$, $d$ even, of maximum size such that $Y \cap Z \neq Y$. 
    Let $G \in Y$. Let $\tilde{Y}$ be the set of the $q^{\binom{d}{2}}$ generators of $Y$ disjoint to $G$ (see Proposition \ref{prop_q0_inmbs}).
    \begin{enumerate}[(a)]
     \item There exists a hyperplane $h$ of type $Q^+(2d-1, q)$ such that $G, \tilde{Y} \subseteq h$.
     \item If $\tilde{G} \in Y \setminus h$, then $\tilde{G}$ meets all elements of $\tilde{Y}$ non-trivially.
     \item If $\tilde{G} \in Y$ and $\dim(\tilde{G} \cap H) = d-2$ for any $H \in \tilde{Y}$, then $\tilde{G}$ and the $q^{\binom{d}{2}}$ generators disjoint to $\tilde{G}$ are in $h$.
    \end{enumerate}
\end{lemma}
\begin{proof}
  By Proposition \ref{prop_q0_inmbs}, a generator $G \in Y$ is disjoint to $q^{\binom{d}{2}}$ generators of $Y$. 
  Let $H \in \tilde{Y}$.
  By Lemma \ref{lem_q0inhyp_pre2}, $h := \langle G, H \rangle$ has type $Q^+(2d-1, q)$. 
  We shall show $\tilde{Y} \subseteq h$.

  Suppose to the contrary that there exists a generator $\tilde{H} \in \tilde{Y}$ not in $h$.
  We define $Z_G$ in \eqref{eq_lem_nmb_dm1z} by 
  \begin{align*}
   Z_G &= Z_{G, H} \\
   &= \{ z_i ~|~ i \in \{ 1, \ldots, \gaussm{d}\}\} \subseteq Z.
  \end{align*}
  Set
  \begin{align*}
   \scrP := \{ \tilde{H} \cap z_i ~|~ i \in \{ 1, \ldots, \gaussm{d}\} \}.
  \end{align*}
  The generators of $Z_G$ intersect $G$ in a hyperplane of $G$, hence $\scrP$ is a set of points.
  For $i \neq j$ we have $G \subseteq \langle z_i \cap G, z_j \cap G \rangle^\perp$, so $(z_i \cap z_j) \setminus G$ is empty.
  Hence, $|\scrP| = |Z_G| = \gaussm{d}$.
  Furthermore, $Z_G \subseteq h$, so $\scrP \subseteq \tilde{H} \cap h$.
  Hence,
  \begin{align*}
    \gaussm{d} = |\scrP| \leq |\tilde{H} \cap h| = \gaussm{d-1}.
  \end{align*}
  This is a contradiction.
  Thus, $\tilde{Y} \subseteq h$.
  This proves (a).
  
  Assume that there exists a generator $\tilde{G} \in Y \setminus h$ with $\dim(H \cap \tilde{G})=0$.
  Then, by (a), $H$ is only disjoint to generators of $Y$ in $\langle \tilde{G}, H \rangle$.
  Hence, $\langle \tilde{G}, H \rangle = \langle G, H \rangle$.
  This contradicts $\tilde{G} \in Y \setminus h$.
  This proves (b).
  
  Suppose that there exists a generator $\tilde{G} \in Y \setminus h$ with $\dim(H \cap \tilde{G})=d-2$.
  By (b) and Proposition \ref{prop_q0_inmbs}, $\dim(G \cap \tilde{G}) \geq 2$. Hence,
  \begin{align*}
    \tilde{G} = \langle H \cap \tilde{G}, G \cap \tilde{G} \rangle \subseteq h,
  \end{align*}
  which contradicts $\tilde{G} \nsubseteq h$.
  This shows $\tilde{G} \in \tilde{Y}$.
  By (a), all elements disjoint to $\tilde{G}$ are in $\langle G, \tilde{G} \rangle = h$.
\end{proof}

We need the following bound.

\begin{lemma}\label{lem_bound_gens}
  Let $q \geq 2$. Let $d \geq 1$. Then
  \begin{align*}
    \prod_{i=1}^{d-1} (q^i+1) \leq \frac{2q^d}{q^d+1} \left( q^{\binom{d}{2}} - q^{\binom{d-1}{2}} + 1\right) + q^{\binom{d-2}{2} + 2(d-2)}.
  \end{align*}
\end{lemma}
\begin{proof}
  We will prove the assertion by induction over $d$.
  It can be easily checked that the assertion is true for $d \leq 4$.
  If the assertion is true for $d \geq 4$, then
  \begin{align*}
    \prod_{i=1}^{d} (q^i+1) &\leq (q^d+1) \left(\frac{2q^d}{q^d+1} \left( q^{\binom{d}{2}} - q^{\binom{d-1}{2}} + 1\right) + q^{\binom{d-2}{2} + 2(d-2)} \right) \\
    &\stackrel{(*)}{\leq} \frac{2q^{d+1}}{q^{d+1}+1} \left(q^{\binom{d+1}{2}} - q^{\binom{d}{2}} + 1\right) + q^{\binom{d-1}{2} + 2(d-1)}.
  \end{align*}
  The difference between the right hand side of (*) and the left hand side of (*) equals
  \begin{align*}
   \frac{q^{-\frac{3d}{2}-1}}{q^d+1} & ( 2{q}^{\frac{{d}^{2}}{2}+2d+3}-2{q}^{\frac{{d}^{2}}{2}+2d+2}-3{q}^{\frac{{d}^{2}}{2}+2d+1}\\
   &+2{q}^{\frac{{d}^{2}}{2}+d+2}-{q}^{\frac{{d}^{2}}{2}+d}-2{q}^{\frac{7d}{2}+2}+2{q}^{\frac{5d}{2}+2}-2{q}^{\frac{5d}{2}+1})
   \\
  \end{align*}
  which is a positive expression for $q \geq 2$ and $d \geq 4$.
\end{proof}

\begin{prop}\label{lem_q0inhyp}
  Let $(Y, Z)$ be a cross-intersecting EKR set of $Q(2d, q)$ of maximum size such that $Y \cap Z \neq Y$. Then there exists a hyperplane $h$ such that $Y, Z \subseteq h$.
\end{prop}
\begin{proof}
  In the view of Lemma \ref{lem_q0inhyp_pre1}, we find a hyperplane $h$ that contains $G$, 
  the $q^{\binom{d}{2}}$ generators of $Y$ disjoint to $G$, and, by Proposition \ref{prop_q0_inmbs} and 
  Lemma \ref{lem_q0inhyp_pre1} (d), the $\gauss{d}{2} q^{\binom{d-2}{2}}$ generators which meet $G$ in dimension $d-2$.
  
  Suppose that there exists an element $G \in Y$ that is not in $h$.
  Then $G$ and the $q^{\binom{d}{2}}$ elements of $Y$ disjoint to $G$ lie in a second hyperplane $h' \neq h$ by Lemma \ref{lem_q0inhyp_pre1}.
  Lemma \ref{lem_q0inhyp_pre1} makes it clear that the at least $q^{\binom{d}{2}} + 1$ generators of $Y$ in $h$ are different to the at least $q^{\binom{d}{2}} + \gauss{d}{2} q^{\binom{d-2}{2}} + 1$ generators of $Y$ in $h'$.
  
  Hence, 
  \begin{align*}
    |Y| \geq 2\left(q^{\binom{d}{2}} + 1\right)  + \gauss{d}{2} q^{\binom{d-2}{2}}.
  \end{align*}
  According to Theorem \ref{thm_es},
  \begin{align*}
    |Y| = \prod_{i=1}^{d-1} (q^i+1).
  \end{align*}
  This contradicts Lemma \ref{lem_bound_gens}.
\end{proof}

\begin{satz}\label{satz_ci_ekr_q0}
  Let $(Y, Z)$ be a cross-intersecting EKR set of $Q(2d, q)$, or $W(2d-1, q)$, $q$ even, of maximum size such that $Y \cap Z \neq Y$.
  Then either $Y=Z$ and $Y$ is an EKR set, or $d$ even and $Y \cup Z$ are the generators of a subgeometry isomorphic to $Q^+(2d-1, q)$.
\end{satz}
\begin{proof}
  First consider $Q(2d, q)$. 
  By Proposition \ref{lem_q0inhyp}, $Y, Z \subseteq h$ for some hyperplane $h$ isomorphic to $Q^+(2d-1, q)$ if not $Y=Z$.
  Hence, $(Y, Z)$ is a cross-intersecting set of $Q^+(2d-1, q)$ of maximum size.
  These sets were classified in Theorem \ref{thm_ci_ekr_qp}.
  
  The part of the assertion for $W(2d-1, q)$, $q$ even, follows, since then $Q(2d, q)$ and $W(2d-1, q)$ are isomorphic for $q$ even.
\end{proof}

\subsubsection{The Symplectic Polar Space $W(2d-1, q)$, $d$ even, $q$ odd}

Similar to \cite{MR2755082} we use the following property of $W(2d-1, q)$, $d$ even (\cite[Theorem 34]{MR2755082}):

\begin{satz}\label{thm_wq_odd_3lines}
  Let $\ell_1, \ell_2, \ell_3$ be three pairwise disjoint lines of $W(3, q)$, $q$ odd. 
  Then the number of lines meeting $\ell_1, \ell_2, \ell_3$ is $0$ or $2$.
\end{satz}

\begin{satz}
  Let $(Y, Z)$ be a cross-intersecting EKR set of maximum size of $W(2d-1, q)$, $d$ even, $q$ odd. Then $Y=Z$.
\end{satz}
\begin{proof}
  Suppose to the contrary that $Y \cap Z \neq Y$.
  By Proposition \ref{prop_q0_inmbs}, we can find two disjoint generators $G$ and $H$ in $Y$.
  Again by Proposition \ref{prop_q0_inmbs}, there are exactly $q \gaussm{d} \gaussm{d-1}/(q+1)$ generators $Y' \subseteq Y$ which meet $G$ in a subspace of dimension $d-2$.
  The generator $G$ has $\gaussm{d} \gaussm{d-1}/(q+1)$ subspaces of dimension $d-2$.
  Hence, we find a subspace $\ell \subseteq G$ of dimension $d-2$ such that $\ell$ is contained in $q$ elements of $Y'$.
  Since $q$ is odd, there are at least three elements $y_1, y_2, y_3$ of $Y$ through $\ell$.
  
  Consider the quotient geometry $W_3$ of $\ell$ isomorphic to $W(3, q)$ and the projection of the elements of $Y$ and $Z$ onto $W_3$ from $\ell$.
  Since elements of $Y$ do not meet each other in dimension $d-1$ by Proposition \ref{prop_q0_inmbs}, $y_1, y_2, y_3$ are three disjoint lines in $W_3$ after projection.
  Lemma \ref{lem_nmb_dm1z} says that the $\gaussm{d}$ generators $Z' \subseteq Z$ which meet $H$ in dimension $d-1$ also meet $G$ in $\gaussm{d}$ pairwise different points.
  Hence there are at least $3$ generators in $Z'$ which are projected onto three different lines on $W_3$.
  These three lines have to meet the projections of $y_1$, $y_2$, and $y_3$, since $(Y, Z)$ is an cross-intersecting EKR set.
  By Theorem \ref{thm_wq_odd_3lines} this is not possible.
  Contradiction.
\end{proof}

\section{The Hermitian Polar Space $H(2d-1, q^2)$}\label{sec_hermitian}

It is well-known (see for example \cite{MR2312327}) that the linear programming 
bound given in \cite{fi_km_ekr_in_hermitians} can be reformulated as a weighted 
Hoffman bound. Hence, Lemma \ref{lem_efp} is applicable if $d > 1$. 
The original bound on EKR sets on $H(2d-1, q^2)$ is as follows.

\begin{satz}[\cite{fi_km_ekr_in_hermitians}]\label{lem_h2dp1_normal}
  Let $Y$ be a EKR set of $H(2d-1, q^2)$ with $d > 1$ odd. Then 
  \begin{align*}
    |Y| \le \frac{nq^{d-1} - f_1(q^{d-1}-1)\left(1 - c\right)}{q^{2d-1}+q^{d-1}+f_1(q^{d-1}-1)c}\approx q^{d^2-2d+2},
  \end{align*}
  where $n=\prod_{i=0}^{d-1} (q^{2i+1}+1)$, $f_1=q^2 \gaussm{d}_{q^2} \frac{q^{2d-3}+1}{q+1}$ and $c=\frac{q^2-q-1+q^{-2d+3}}{q^{2d}-1}$.
\end{satz}

The result by Luz \cite{MR2312327} which shows
that the linear programming bound is a special case of the weighted Hoffman bound\footnote{This seems to be part of the mathematical folklore for a long time, but the author is not aware of any source older than \cite{MR2312327}.}
also holds for cross-intersecting EKR sets, but we feel that we should show this 
directly, since the transition from the linear programming technique used in 
\cite{fi_km_ekr_in_hermitians} to the weighted Hoffman bound is not obvious.
We shall prove a cross-intersecting result similar to \cite{fi_km_ekr_in_hermitians}
in the following.

\begin{satz}\label{lem_h2dp1}
  Let $(Y, Z)$ be a cross-intersecting EKR set of $H(2d-1, q^2)$ with $d > 1$. Then 
  \begin{align*}
    \sqrt{|Y|\cdot |Z|} \le \frac{n\lambda_b}{\lambda_b-k} \approx q^{d^2-2d+2},
  \end{align*}
  where $n=\prod_{i=0}^{d-1} (q^{2i+1}+1)$, $\lambda_b = -q^{(d-1)^2} - \alpha \left( 1 -  f_1 \frac{1 - c}{n}\right)$,
  $k = q^{d^2} + \alpha f_1 \left( c + \frac{1-c}{n} \right)$,
  $f_1=q^2 \gaussm{d}_{q^2} \frac{q^{2d-3}+1}{q+1}$, $c=\frac{q^2-q-1+q^{-2d+3}}{q^{2d}-1}$,
  and
\begin{align*}
  \alpha = \begin{cases}
            q^{d(d-1)} + q^{(d-1)^2} & \text{ if $d$ odd,}\\
            \frac{nq^{d^2-d} - n q^{(d-1)^2}}{n + (2c-2)f_1} & \text{ if $d$ even.}
           \end{cases}
\end{align*}
\end{satz}

\begin{proof}
Let $d > 1$.
Let $A_d$ be the disjointness matrix as defined in Section \ref{sec_assoc}. Consider
the matrix $A$ defined as
\begin{align*}
  A = A_d - \alpha E_1 + \frac{\alpha f_1 c}{n} J + \alpha f_1 \frac{1 - c}{n} I.
\end{align*}

\paragraph*{Claim 1} Our first claim is that $A$ is a extended weight adjacency matrix. By Section \ref{sec_assoc},
it is clear that the entry $(x, y)$ of $E_1$ equals $Q_{i,1}/n$ if $x$ and $y$ meet
in codimension $i$. It was shown in \cite[Equations (6)--(11)]{fi_km_ekr_in_hermitians} 
that the following holds (note that the equation in \cite{fi_km_ekr_in_hermitians} do not depend on $d$ odd):
\begin{multicols}{2}
\begin{enumerate}[(a)]
 \item $Q_{0, 1} = f_1$, 
 \item $Q_{d-1, 1} = f_1 c$,
 \item $Q_{s, 1} \geq f_1 c$ if $s < d$,
 \item $Q_{d, 1} < 0$.
\end{enumerate}
\end{multicols}
Hence, the entry $(x, y)$ of the matrix $A$ is $0$ if $x = y$, it is less or equal to
zero if $1 \leq \codim(x \cap y) \leq d-1$, and it is larger than $1$ if $x$ and $y$ are disjoint.
This shows that $A$ is an extended weight adjacency matrix of the disjointness graph of generators.

\paragraph*{Claim 2} Our second claim is that one of the second absolute largest 
eigenvalues of $A$ is
\begin{align*}
  -q^{(d-1)^2} - \alpha \left( 1 -  f_1 \frac{1 - c}{n}\right),
  \intertext{ and that }
  k = q^{d^2} + \alpha f_1 \left( c + \frac{1-c}{n} \right).
\end{align*}
By \eqref{eq_ev_disj}, the eigenvalues of $A$ are 
\begin{align*}
  &q^{d^2} + \alpha f_1 \left( c + \frac{1-c}{n} \right) \text{ for } \langle j \rangle,\\
  &-q^{(d-1)^2} - \alpha \left( 1 -  f_1 \frac{1 - c}{n}\right) \text{ for } W_1,\\
  &(-1)^r q^{(d-r)^2 + r(r-1)} + \alpha f_1 \frac{1 - c}{n}\text{ for } W_r \text{ with } 1 < r < d,\\
  &(-1)^d q^{d(d-1)} + \alpha f_1 \frac{1 - c}{n} \text{ for } W_d.
\end{align*}
An simple calculation shows that 
$-q^{(d-1)^2} - \alpha \left( 1 -  f_1 \frac{1 - c}{n}\right) = (-1)^d (q^{d(d-1)} - \alpha f_1 \frac{1 - c}{n})$
is the second largest absolute eigenvalue. This proves our claim. 

Now we can apply Lemma \ref{lem_efp} with these values. Note that $k$ has approximately
size $q^{d^2+d-2}$, the second largest absolute eigenvalue $\lambda_b$ has approximately size
$q^{d(d-1)}$, and $n$ has approximately size $q^{d^2}$. Therefore,
\begin{align*}
  \frac{n\lambda_b}{\lambda_b-k}
\end{align*}
has approximately size $q^{d^2-2d+2}$. 
\end{proof}

Note that the normal adjacency matrix of the graph only
yields $q^{d^2-d}$ as an upper bound, so this improves the bound significantly.

For the sake of completeness we want to mention the cross-intersecting EKR sets for $d=2$ is 
We will do this after providing a general geometrical results on (maximal) cross-intersecting EKR sets, where we call an (cross-intersecting) EKR set $(Y, Z)$ \emph{maximal} if there exists no generator $x$ such that $(Y \cup \{ x \}, Z)$ or $(Y, Z \cup \{ x \})$ is an cross-intersecting EKR set.

\begin{lemma}\label{lem_zy_intersect}
  Let $(Y, Z)$ be a maximal cross-intersecting EKR set in a finite classical polar space of rank $d$. 
  If two distinct elements $y_1, y_2 \in Y$ meet in a subspace of dimension $d-1$, then all elements of $Z$ meet this subspace in at least a point.
\end{lemma}
\begin{proof}
 Assume that there exists a generator $z$ which meets $y_1$ and $y_2$ in points $P, Q$ not in $y_1 \cap y_2$. 
   Then $\langle P, Q, y_1 \cap y_2 \rangle$ is a totally isotropic subspace of dimension $d+1$. Contradiction.
\end{proof}

Lemma \ref{lem_efp} yields
\begin{align*}
  \frac{(q+1) (q^3+1)}{q^2+1}
\end{align*}
as an upper bound for $H(3, q^2)$. This bound is not sharp as the following trivial results shows.

\begin{satz}\label{thm_h3}
  Let $(Y, Z)$ be a maximal cross-intersecting EKR set of $H(3, q^2)$ with $|Y| \geq |Z|$. Then one of the following cases occurs:
  \begin{enumerate}[(a)]
   \item The set $Y$ is the set of all lines of $H(3, q^2)$, and $Z = \emptyset$. Here $|Y| \cdot |Z| = 0$.\label{thm_h3_a}
   \item The set $Y$ is the set of all lines meeting a fixed line $\ell$ in at least a point, and $Z = \{ \ell \}$. Here $|Y| \cdot |Z| = (q^2+1)q + 1$.\label{thm_h3_b}
   \item The set $Y$ is the set of all lines on a fixed point $P$, and $Y = Z$. Here $|Y| \cdot |Z| = (q+1)^2$.\label{thm_h3_c}
   \item The set $Y$ is the set of lines meeting two disjoint lines $\ell_1, \ell_2$, and $Z = \{ \ell_1, \ell_2\}$. Here $|Y| \cdot |Z| = 2(q^2+1)$.\label{thm_h3_d}
   \item The set $Y$ is the set of lines meeting three disjoint lines $\ell_1, \ell_2, \ell_3$, and $Z$ is the set of all $q+1$ lines meeting the lines of $Y$. Here $|Y| \cdot |Z| = (q+1)^2$.\label{thm_h3_e}
  \end{enumerate}
\end{satz}
\begin{proof}
  Assume that \eqref{thm_h3_a} does not occur. 
  
  By Lemma \ref{lem_zy_intersect}, as soon as two elements of $Y$ meet in a point $P$, then all elements of $Z$ contain $P$. 
  Hence, \eqref{thm_h3_b} occurs or at least $2$ elements of $Z$ meet in $P$. Hence, all elements of $Y$ contain $P$ by Lemma \ref{lem_zy_intersect}.
  This is case \eqref{thm_h3_c}.
  
  So assume that $Y$ and vice-versa $Z$ only consist of disjoint lines. 
  If there are two lines $\ell_1, \ell_2 \subseteq Z$, then there are $q^2+1$ (disjoint) lines $L$ meeting $\ell_1$ and $\ell_2$ in a point (hence $|Y| \leq q^2+1$). 
  If more than $q+1$ of these lines meet $\ell_1$ (hence $|Y| > q+1$), then $Z$ contains at most two lines, since in $H(3, q^2)$ exactly $q+1$ lines meet $3$ pairwise disjoint lines in a point.
  This yields \eqref{thm_h3_d}.
  If $|Z| \geq 3$, then $|Y| \leq q+1$ by the previous argument.
  We may assume $|Y| \geq 3$.
  Then it is well-known that there are exactly $q+1$ lines meeting the $q+1$ lines of $Y$. 
  Hence, we can add these lines and then $Z$ is maximal.
  This yields \eqref{thm_h3_e}.
\end{proof}

The author tried to prove that the maximum cross-intersecting EKR set of $H(5, q^2)$ is the unique EKR of maximum size given in \cite{MR2755082}, but aborted this attempt after he got lost in too many case distinctions.
This EKR set of all generators meeting a fixed generator in at least a line is the largest cross-intersecting EKR set known to the author and has size $q^5+q^3+q+1$.
The largest example known to the author for $H(7, q^2)$ is the following.

\begin{Example}\label{ex_largst_h7}
  Let $G$ be a generator of $H(7, q^2)$.
  Let $Y$ be the set of all generators that meet $G$ in at least a $2$-space.
  Let $Z$ be the set of all generators that meet $G$ in at least a $3$-space.
  Then $(Y, Z)$ is a cross-intersecting EKR set.
\end{Example}
\begin{proof}
 A generator of $H(7, q^2)$ is a $4$-space. 
 A plane and a line of a $4$-space meet pairwise in at least a point.
 Hence, $(Y, Z)$ is a cross-intersecting EKR set.
\end{proof}

In this example $Y$ has
\begin{align*}
  1 + q + q^3 + q^4 + q^5 + q^6 + q^7 + 2q^8 + q^{10} + q^{12}
\intertext{elements, $Z$ has}
  1 + q + q^3 + q^5 + q^7
\end{align*}
elements, so in total the cross-intersecting EKR set has size
\begin{align*}
  \sqrt{|Y|\cdot |Z|} \approx q^{19/2}.
\end{align*}
The bound given in Theorem \ref{lem_h2dp1} for this case is approximately $q^{10}$.
For $H(2d-1, q^2)$, $d > 4$, the largest example known to the author is the EKR set 
of all generators on a fixed point. The author assumes that the largest known examples
are also the largest examples.

\section{Summary}

We summarize our results in the following table. 
We only list the cases, where cross-intersecting EKR sets of maximum size are not necessarily EKR sets.
The table includes the size of the largest known example if it is not known if 
the best known bound does not seem to be tight.

\begin{footnotesize}
\begin{center}
\begin{tabular}{p{2.5cm}lp{3cm}l}\toprule
Polar Space & Maximum Size $\sqrt{|Y|\cdot |Z|}$ & Largest (known) Examples & Reference \\\toprule
$Q^+(2d-1, q)$, $d$ odd & $n/2$ & $Y$ latins, $Z$ greeks & Th. \ref{thm_ci_ekr_qp}\\
$Q(2d, q)$, $d$ odd & $(q+1) \cdot \ldots \cdot (q^{d-1}+1)$ & $Y$ latins and $Z$ greeks of a $Q^+(2d+1, q)$, or $Y=Z$ EKR set & Th. \ref{satz_ci_ekr_q0} \\
$W(2d-1, q)$, $d$ odd, $q$ even & $(q+1) \cdot \ldots \cdot (q^{d-1}+1)$ & see $Q(2d, q)$ & Th. \ref{satz_ci_ekr_q0} \\
$H(3, q^2)$ & $q^3+q+1$ & Th. \ref{thm_h3} & Th. \ref{thm_h3} \\
$H(5, q^2)$ & $\lessapprox q^{5}$ & largest EKR set, size $\approx q^{5}$ & Th. \ref{lem_h2dp1} \\
$H(7, q^2)$ & $\lessapprox q^{10}$ & Example \ref{ex_largst_h7}, size $\approx q^{19/2}$ & Th. \ref{lem_h2dp1} \\
$H(2d-1, q^2)$, $d > 1$ & $\lessapprox q^{(d-1)^2+1}$ & all generators on a point, size $\approx q^{(d-1)^2}$ & Th. \ref{lem_h2dp1}\\
\end{tabular}
\end{center}
\end{footnotesize}

\bibliographystyle{plain}
\bibliography{literatur}

\appendix

% \closeoutputstream{todo_list}

\end{document}